\documentclass[a4paper,12pt]{article}
\usepackage{amssymb,amsmath,amsthm,latexsym}
\usepackage[pdftex,bookmarks,colorlinks=false]{hyperref}
\usepackage[hmargin=1.2in,vmargin=1.2in]{geometry}

\newtheorem{theorem}{Theorem}[section]

\newtheorem{corollary}[theorem] {Corollary}
\newtheorem{definition}[theorem]{Definition}

\newtheorem{lemma} [theorem]{Lemma}

\newtheorem{proposition}[theorem]{Proposition}
\newtheorem{remark}[theorem]{Remark}

\title{\textbf{\sc On Certain Arithmetic Integer Additive Set-Indexers of Graphs}}
\author{{\bf N K Sudev} $^{{1},{\ast}}$ and {\bf K A Germina $^{2}$}
\\ \\
$^{1}${\small Department of Mathematics}\\ {\small Vidya Academy of Science \& Technology} \\ {\small  Thalakkottukara, Thrissur - 680501, Kerala, India.}\\ {\small email: {\em sudevnk@gmail.com}}
\\ \vspace{0.3cm}
$^{\ast}$ {\small Corresponding author.}
\\
$^{2}${\small Department of Mathematics} \\ {\small School of Mathematical \& Physical Sciences} \\ {\small Central University of Kerala, Kasaragod - 671316, Kerala, India.}\\ {\small email: {\em srgerminaka@gmail.com}}
}
\date{}
\begin{document}
\maketitle

\begin{abstract}
Let $\mathbb{N}_0$ denote the set of all non-negative integers and $\mathcal{P}(\mathbb{N}_0)$ be its power set. An integer additive set-indexer (IASI) of a graph $G$ is an injective function $f:V(G)\to \mathcal{P}(\mathbb{N}_0)$ such that the induced function $f^+:E(G) \to \mathcal{P}(\mathbb{N}_0)$ defined by $f^+ (uv) = f(u)+ f(v)$ is also injective, where $\mathbb{N}_0$ is the set of all non-negative integers. A graph $G$ which admits an IASI is called an IASI graph. An IASI of a graph $G$ is said to be an arithmetic IASI if the elements of the set-labels of all vertices and edges of $G$ are in arithmetic progressions. In this paper, we discuss about two special types of arithmetic IASIs. 
\end{abstract}
\vspace{0.05cm}


\section{Introduction}

\subsection{Preliminaries on integer additive set-indexers}

For all  terms and definitions, not defined in this paper, we refer to \cite{FH} and for more about graph labeling, we refer to \cite{JAG}. Unless mentioned otherwise, all graphs considered here are simple, finite and have no isolated vertices.

The sum set of two sets $A$ and $B$, denoted by  $A+B$, is defined as $A + B = \{a+b: a \in A, b \in B\}$.  If at least one of two sets $A$ and $B$ is countably infinite, then their sum set $A+B$ will also be countably infinite. Hence, all sets mentioned in this paper are finite sets. We denote the cardinality of a set $A$ by $|A|$. Using the concepts of the sum set of two sets, the notion of an integer additive set-indexer of a given graph $G$ is defined in \cite{GA} as follows.

Let $\mathbb{N}_0$ denote the set of all non-negative integers and $\mathcal{P}(\mathbb{N}_0)$ be its power set. An {\em integer additive set-indexer} (IASI, in short) of a given graph $G$ is an injective function $f:V(G)\to \mathcal{P}(\mathbb{N}_0)$ such that the induced function $f^{+}:E(G) \to \mathcal{P}(\mathbb{N}_0)$ defined by $f^{+} (uv) = f(u)+ f(v)$ is also injective.  A graph $G$ which admits an integer additive set-indexer is called an {\em integer additive set-indexed graph} (IASI-graph). 

An IASI is said to be a {\em $k$-uniform IASI} if $|f^{+}(e)| = k$ for all $e\in E(G)$. That is, a connected graph $G$ is said to have a $k$-uniform IASI if all of its edges have the same set-indexing number $k$.

The cardinality of the labeling set of an element (vertex or edge) of a graph $G$ is called the {\em set-indexing number} of that element. 

The vertex set $V$ of a graph $G$ is defined to be {\em $l$-uniformly set-indexed}, if all the vertices of $G$ have the same set-indexing number $l$.

Let $f$ be an IASI defined on a graph $G$ and let $u, v$ be any two adjacent vertices in $G$. Two ordered pairs $(a,b)$ and $(c,d)$ in $f(u)\times f(v)$ are said to be {\em compatible} if $a+b=c+d$. If $(a,b)$ and $(c,d)$ are compatible, then we write $(a,b)\sim (c,d)$. Clearly, $\sim$ is an equivalence relation.

A {\em compatibility class} of an ordered pair $(a,b)$ in $f(u)\times f(v)$ with respect to the integer $k=a+b$ is the subset of $f(u)\times f(v)$ defined by $\{(c,d)\in f(u)\times f(v):(a,b)\sim (c,d)\}$ and is denoted by $\mathsf{C}_k$. Since $f(u)$ and $f(v)$ are finite sets, then each compatibility class $\mathsf{C}_k$ in $f(u)\times f(v)$ contains finite number of elements. 

It is to be noted that no compatibility class in $f(u)\times f(v)$ can be non-empty. If a compatibility class $\mathsf{C}_k$ contains only one element, then it is called a {\em trivial class}. A compatibility class $\mathsf{C}_k$ that contains maximum number of elements are called a {\em maximal compatibility class}.

\begin{lemma}\label{L-CardCC}
\cite{GS0} For a compatibility class $\mathsf{C}_k$ in $f(u)\times f(v)$, we have $1\le |\mathsf{C}_k| \le \min\,(|f(u)|,\,|f(v)|)$. 
\end{lemma}

A compatibility class which contain the highest possible number of elements is called {\em saturated class}. That is, the cardinality of a saturated class in $f(u)\times f(v)$ is $\min(|f(u)|,|f(v)|)$. 

It is to be noted that all saturated classes in $f(u)\times f(v)$ are maximal compatible classes, but a maximal compatible class need not be a saturated class of $f(u)\times f(v)$. That is, the existence of a saturated class depends on the nature of elements in the set-labels $f(u)$ and $f(v)$.

Based on the relation between the set-indexing numbers of an edge and its end vertices in $G$, the following notion is introduced in \cite{GS2}.

A {\em strong IASI} is an IASI $f$ such that $|f^{+}(uv)|=|f(u)|\,|f(v)|$ for all $u,v\in V(G)$. A graph which admits a  strong IASI may be called a {\em strong IASI graph}. A  strong  IASI is said to be  {\em strongly uniform IASI} if $|f^{+}(uv)|=k$, for all $u,v\in V(G)$ and for some positive integer $k$.

\subsection{Arithmetic Integer Additive Set-Indexers}

By the term, an arithmetically progressive set, (AP-set, in short), we mean a set whose elements are in an arithmetic progression. In this context, since the set-labels of the elements of $G$ need to be AP-sets, we take the sets having at least three elements for labeling the vertices of a given graph $G$. 

The common difference of the set-label of an element of a graph $G$ is called the {\em deterministic index} of that element. The {\em deterministic ratio} of an edge $e$ of $G$ is the ratio, greater than or equal to $1$,  between the deterministic indices of its end vertices.

A study about the graphs whose elements are labeled by AP-sets, has been done in \cite{GS7} and proposed the following notions and results.

Let $f:V(G)\to \mathcal{P}(\mathbb{N}_0)$ be an IASI on $G$. For any vertex $v$ of $G$, if $f(v)$ is an AP-set, then $f$ is called a {\em vertex-arithmetic IASI} of $G$. For an IASI $f$ of $G$, if $f^+(e)$ is an AP-set, for all $e\in E(G)$, then $f$ is called an {\em edge-arithmetic IASI} of $G$. A graph that admits a vertex-arithmetic IASI (or an edge-arithmetic IASI) is called a {\em vertex-arithmetic IASI graph} (or an {\em edge-arithmetic IASI graph}). 

An IASI is said to be an {\em arithmetic integer additive set-indexer} if it is both vertex-arithmetic and edge-arithmetic. That is, an arithmetic IASI of a given graph $G$ is an IASI $f$, under which the set-labels of all elements of $G$ are AP-sets. A graph that admits an arithmetic IASI is called an {\em arithmetic IASI graph}. 

The admissibility of an arithmetic IASI by a graph is established in the following theorem.

\begin{theorem}\label{T-AIASI-g}
\cite{GS7} A graph $G$ admits an arithmetic IASI $f$ if and only if $f$ is a vertex arithmetic IASI and the deterministic ratio any edge of $G$ is a positive integer, which is less than or equal to the set-indexing number of its end vertex having smaller deterministic index. 
\end{theorem}

\noindent In other words, if $v_i$ and $v_j$ are two adjacent vertices of $G$, with deterministic indices $d_i$ and $d_j$ respectively with respect to an IASI $f$ of $G$, where $d_i\le d_j$, then $f$ is an arithmetic IASI if and only if $d_j=k\,d_i$, where $k$ is a positive integer such that $1\le k \le |f(v_i)|$.

In this paper, we study the characteristics given graphs, the set-labels of whose vertices and edges are AP-sets, with certain properties.

\section{Isoarithmetic IASI of Graphs}

If two AP-sets have the same common difference $d$, then their sum set is also an AP-set with the same common difference $d$. In view of this property, we introduce the following notion.

\begin{definition}{\rm
Let $f$ be an arithmetic IASI defined on a given graph $G$. If all the elements of $G$ have the same deterministic index under $f$, then $f$ is said to be an {\em isoarithmetic IASI} of $G$. A graph which admits an isoarithmetic IASI is called an {\em isoarithmetic IASI graph}.}
\end{definition}

Note that if an IASI $f$ of a graph $G$ is an isoarithmetic IASI, then the set-labels of all elements of $G$ are AP-sets with the same common difference and the deterministic ratio of every edge of $G$ is $1$.

\begin{definition}{\rm
Let $f$ be an isoarithmetic IASI of a given graph $G$, under which $V(G)$ is $l$-uniformly set-indexed, then $f$ is called an $l$-uniform {\em $l$-uniform isoarithmetic IASI} of $G$.}
\end{definition}

In the following discussions, we study certain characteristics of isoarithmetic IASI graphs. The following theorem verifies the existence of isoarithmetic IASIs for given graphs.
 
\begin{theorem}
Every graph $G$ admits an isoarithmetic integer additive set-indexer.
\end{theorem}
\begin{proof}
Let $f$ be an IASI defined on a given graph $G$ such that, for any vertex $v_i$ of $G$,  $f(v_i)$ is an AP-set with the same common difference $d$ , where $d>1$ is a non-negative integer. Then, $f^+(v_iv_j)$ is also an AP-set with the same common difference $d$, for all edges $v_iv_j\in E(G)$. Therefore, $f$ is an isoarithmetic IASI of $G$.
\end{proof}

The following result establishes the hereditary nature of the existence of an isoarithmetic IASI of a graph $G$.

\begin{proposition}\label{P-APSL0}
A subgraph of an isoarithmetic IASI graph $G$ admits an (induced) isoarithmetic IASI. That is, the existence of an isoarithmetic IASI is a hereditary property.
\end{proposition}
\begin{proof}
Let $H$ be a subgraph of the graph $G$. Let $f$ be an isoarithmetic IASI of $G$. Then, the restriction $f|_H$ of $f$ to $V(H)$ is an isoarithmetic IASI of $H$. Hence $H$ is also an isoarithmetic IASI graph.
\end{proof}

An interesting question that arises here is about the set-indexing number of edges of an isoarithmetic IASI graph. To proceed in this direction, we need the following result.

\begin{lemma}\label{L-SS-5a}
Let $A$ and $B$ be finite AP-sets of integers having the same common difference $d$. Then, $|A+B|=|A|+|B|-1$.
\end{lemma}

\begin{lemma}\label{L-SS-5b}
\cite{MBN} Let $A$ and $B$ be finite sets of integers with $|A|=k\ge 2, |B|=l\ge 2 $. If $|A+B|=k+l-1$, then $A$ and $B$ are arithmetic progressions with the same common difference.
\end{lemma}

\noindent Invoking the above lemma, we propose the following theorem.

\begin{theorem}\label{T-APSL}
Let $G$ be a graph with an arithmetic IASI $f$ defined on it. Then, $f$ is an isoarithmetic IASI on $G$ if and only if the set-indexing number of every edge of $G$ is one less than the sum of the set-indexing numbers of it end vertices.
\end{theorem}
\begin{proof}
Let $v_i$ and $v_j$ be two adjacent vertices on $G$. Then, $f(v_i)$ and $f(v_j)$ are two AP-sets with cardinalities $m$ and $n$ respectively. Since $f$ is an arithmetic IASI of $G$, $f^+(v_iv_j)$ is also an AP-set.

First, assume that $f$ is an isoarithmetic IASI on $G$. Then, $f(v_i)$ and $f(v_j)$ are AP-sets with the same common difference, say $d$. Then, the set-label of the edge $v_iv_j$ is the set $f^+(v_iv_j)$, which is also an AP-set with the same common difference $d$. Therefore, by Lemma \ref{L-SS-5a}, the set-indexing number of the edge $v_iv_j$ is $m+n-1$.

Conversely, assume that he set-indexing number of every edge of $G$ is one less than the sum of the set-indexing numbers of it end vertices. That is, for any edge $v_iv_j$ in $G$, we have $|f^+(v_iv_j)|=|f(v_i)|+|f(v_j)|-1$. Since $f$ is an arithmetic IASI of $G$, $|f(v_i)|\ge 3 ~~ \forall ~v_i\in V(G)$. Therefore, by Lemma \ref{L-SS-5b}, both $f(v_i)$ and $f(v_j)$ also have the same common difference that of $f^+(v_iv_j)$. Hence, $f$ is an isoarithmetic IASI of $G$.
\end{proof}

\noindent The following theorem is an immediate consequence of Theorem \ref{T-APSL}.

\begin{theorem}\label{T-APSL3}
Let $f$ be an arithmetic IASI defined on a given graph $G$ such that $V(G)$ is $l$-uniformly set-indexed. Then, $f$ is an isoarithmetic IASI of $G$ if and only if $G$ is a $(2l-1)$-uniform IASI graph.
\end{theorem}
\begin{proof}
Let $V(G)$ is $l$-uniformly set-indexed under an arithmetic IASI $f$. Then, we have $|f(v_i)|=|f(v_j)|=l$ for any two (adjacent) vertices of $G$. Then, by Theorem \ref{T-APSL}, $f$ is an isoarithmetic IASI of $G$ if and only if $|f^+(v_iv_j)=2l-1$ for every edge $v_iv_j$ in $G$. 
\end{proof}

The following result addresses the question whether an isoarithmetic IASI could be a strong IASI.

\begin{proposition}\label{P-APSL2}
No isoarithmetic IASI defined on a given graph $G$ can be a strong IASI of $G$.
\end{proposition}
\begin{proof}
Let $f$ be an isoarithmetic IASI of a graph $G$. Then, the set-labels of the vertices of $G$ under $f$ are AP-sets with the same common difference $d$. If possible, let $f$ be a strong IASI. Then, by Theorem \ref{T-APSL}, we have $m+n-1=mn$. This condition holds only when $m=1$ or $n=1$, which is a contradiction to the fact that the set-labels of the elements of $G$ contain at least $3$ elements. Hence, $f$ is not a strong IASI of the graph $G$.
\end{proof}

In view of Proposition \ref{P-APSL2}, for any two adjacent vertices $v_i, v_j\in V(G)$, it can be seen that under an isoarithmetic IASI $f$ on $G$, some compatibility classes in $f(v_i)\times f(v_j)$, contain more than one element. Then, the question about the number of elements in various compatibility classes arises much interest. The following theorem discusses the number of elements in the compatibility classes of $f(v_i)\times f(v_j)$ in $G$.

\begin{theorem}\label{T-NCC}
Let $G$ be a graph which admits an isoarithmetic IASI, say $f$. Then, the number of saturated classes in the Cartesian product of the set-labels of any two adjacent vertices in $G$ is one greater than the difference between cardinality of the set-labels of these vertices. More over, exactly two compatibility classes, other than the saturated classes, have the same cardinality in the Cartesian product of the set-labels of these vertices.

\end{theorem}
\begin{proof}
Let $v_i$ and $v_j$ be two adjacent vertices in $G$. Also, let $|f(v_i)|=m$ and $|f(v_j)|=n$. Without loss of generality, let $m\ge n$. Then, by lemma \ref{L-CardCC}, the maximum cardinality of a compatible class in $f(v_i)\times f(v_j)$ is $n$. 

Let $f(v_i)=\{a, a+d,a+2d,\ldots,a+(m-1)d\}$ and $f(v_j)=\{b, b+d,b+2d,\ldots,b+(n-1)d\}$, where $a$ and $b$ are two positive integers. Consider the set-label of the edge $v_iv_j$ defined by $f^{+}(v_iv_j)=f(v_i)+f(v_j)$. Then, $f^{+}(v_iv_j)=\{a+b,a+b+d,a+b+2d,\ldots,a+b+(m+n-2)d\}$. 

By Lemma \ref{L-CardCC}, a compatibility class can have at most of $n$ elements. Let $r=a+b$. Then, the compatibility classes in $f(v_i)\times f(v_j)$ are given by,
\begin{eqnarray*}
\mathsf{C}_r  & = & \{(a,b)\},\\ 
\mathsf{C}_{r+d}  & = & \{(a+d,b),(a,b+d)\},\\ 
\mathsf{C}_{r+2d}  & = & \{(a+2d,b),(a+d,b+d),(a,b+2d)\},\\ 
\mathsf{C}_{r+3d}  & = & \{(a+3d,b),(a+2d,b+d),(a+d,b+2d),(a,b+3d)\},\\
%\mathsf{C}_{r+4d}  & = & \{(a+4d,b),(a+3d,b+d),(a+2d,b+2d),(a+d,b+3d),\\(a,b+4d)\},\\
........&...&......................................................\\
........&...&.......................................................\\
%\mathsf{C}_{r+(m+n-4)d} & = & \{(a+(n-1)d, b+(m-3)d), (a+(n-2)d,  b+(m-2)d),(a+(n-3)d,  b+(m-1)d)\},\\
\mathsf{C}_{r+(m+n-3)d} & = & \{(a+(n-1)d, b+(m-2)d), (a+(n-2)d,  b+(m-1)d)\},\\  
\mathsf{C}_{r+(m+n-2)d} & = & \{(a+(n-1)d, b+(m-1)d)\}.
\end{eqnarray*}

\noindent Hence, the cardinality of different compatibility classes are,
\begin{eqnarray*}
|\mathsf{C}_r| & = & |\mathsf{C}_{r+(m+n-2)d}|  =  1.\\
|\mathsf{C}_{r+d}|  & = &  |\mathsf{C}_{r+(m+n-3)d}|  =  2.\\
|\mathsf{C}_{r+2d}|  & = &  |\mathsf{C}_{r+(m+n-4)d}|  =  3.\\
|\mathsf{C}_{r+3d}|  & = &  |\mathsf{C}_{r+(m+n-5)d}|  = 4.\\
|\mathsf{C}_{r+4d}|  & = &  |\mathsf{C}_{r+(m+n-6)d}|  = 5.\\
.........&...&...........................\\
.........&...&...........................\\
|\mathsf{C}_{r+(n-2)d}|  & = &  |\mathsf{C}_{r+md}| =  n-1.
\end{eqnarray*}

Then, each of the remaining compatibility classes $\mathsf{C}_{r+(n-1)d}, \mathsf{C}_{r+(n)d}, \ldots, \mathsf{C}_{r+(m-1)d}$ contains $n$ elements, which is the highest number of elements possible in a compatibility class $\mathsf{C_r}$. Hence, all these classes are saturated classes. Therefore, the number of saturated classes in $f(v_i)\times f(v_j)$  = $m+n-1-2(n-1) =m-n+1$.

Also, it can be noted from the above equations that there are exactly two compatibility classes, other than the saturated classes, have the same cardinality $p$, where $1\le p\le n-1$. This completes the proof.
\end{proof}

\begin{corollary}
Let $f$ be an isoarithmetic IASI defined on a graph $G$, under which $V(G)$ is $l$-uniformly set-indexed. Then, there is exactly one saturated class in the Cartesian product of the set-labels of any two adjacent vertices in $G$.
\end{corollary}
\begin{proof}
Let $f$ be an isoarithmetic IASI defined on a graph $G$, under which $V(G)$ is $l$-uniformly set-indexed. Then, for any two adjacent vertices $v_i$ and $v_j$ in $G$, in $|f(v_i)|=|f(v_j)=l$. By Theorem \ref{T-NCC}, the number of saturated classes in $f(v_i)\times f(v_j)$ is $|f(v_i)|-|f(v_j)+1 = 1$.
\end{proof}

Can an isoarithmetic IASI $f$ defined on a given graph $G$ be a uniform IASI of $G$? If so, what are the conditions required for $f$ to be a uniform IASI? The following theorem provides a solution to these questions.
 
\begin{theorem}\label{T-AUIASI1}
An isoarithmetic IASI of a $G$ is a uniform IASI if and only if $V(G)$ is uniformly set-indexed or $G$ is bipartite.
\end{theorem}
\begin{proof}
Let $f$ be an isoarithmetic IASI of a graph $G$. If $V(G)$ is $l$-uniformly set-indexed, then by Theorem \ref{T-APSL3}, $G$ is $(2l-1)$-uniform IASI. Now, assume that $V(G)$ is not uniformly set-indexed. Then, for at least one edge of $G$, say $v_iv_j$, $|f(v_i)|\neq |f(v_j)|$. Let $G$ is bipartite with bipartition $(X,Y)$. Label the vertices of $X$ by distinct $m$-element AP-sets having the common difference $d$ and the vertices of $Y$ by distinct $n$-element AP-sets with the same common difference $d$. Then, by Theorem \ref{T-APSL}, every edge of $G$ has the set-indexing number $m+n-1$. That is, $f$ is $(m+n-1)$-uniform IASI.

Conversely, assume that $f$ is an $r$-uniform IASI of a connected graph $G$. If $V(G)$ is uniformly set-indexed, the proof is complete. Hence, assume that $V(G)$ is not uniformly set-indexed. Since $G$ is connected, there exist a unique pair of distinct positive integers $m$ and $n$ such that $r=m+n-1$ and every edge of $G$ has one vertex with set-indexing number $m$ and other end vertex with set-indexing number $n$. Let $X$ and $Y$ be the sets of all vertices of $G$ with set-indexing number $m$ and  $n$ respectively. Let $v_i\in X$. Then, $v_iv_j\in E(G)\implies f^+(v_iv_j)=m+n-1\implies v_j\in Y$. Similarly, for $v_j\in Y, v_jv_k\in E(G)\implies f^+(v_kv_j)=m+n-1\implies v_k\in X$. Therefore, $(X,Y)$ is a bipartition of $G$. 
\end{proof}

In view of Theorem \ref{T-AUIASI1}, it is natural to enquire whether an isoarithmetic IASI of a disconnected graph $G$ can be a uniform IASI and to determine the conditions, if exist,  required for an isoarithmetic IASI of such a graph $G$ to be a uniform IASI? Let us establish a solution to all these questions in the following theorem. 

\begin{theorem}\label{T-AUIASI2}
An isoarithmetic IASI $f$ of a graph $G$ is an $r$-uniform IASI if and only if every component $G$ is either bipartite or its vertex set is $l$-uniformly set-indexed, where $l=\frac{1}{2}(r+1)$.
\end{theorem}
\begin{proof}
Let $G$ be a graph with $q$ components, say $G_1, G_2,\ldots, G_q$ and $f$ be an isoarithmetic IASI defined on $G$. Let $f$ be an $r$-uniform IASI on $G$. Since each $G_i$ is a subgraph of $G$, by Proposition \ref{P-APSL0}, a restriction $f_i$ of $f$ to $V(G_i)$ induces an isoarithmetic IASI on $G_i$, which is also an $r$-uniform IASI on $G_i$. Since $G_i$ is a connected graph, by Theorem \ref{T-AUIASI1}, $G_i$ is a bipartite graph or $V(G_i)$ is $l$-uniformly set-indexed.

Conversely, assume that every component $G$ is either bipartite or its vertex set is $l$-uniformly set-indexed. If the vertex sets of all components of $G$ are $l$-uniformly set-indexed, then $V(G)$ will also be $l$-uniformly set-indexed. Then by Theorem \ref{T-AUIASI1}, the isoarithmetic IASI $f$ will be a uniform IASI of $G$. If for a component $G_i$ of $G$, $V(G_i)$ is not uniformly indexed, then $G_i$ is a bipartite graph with bipartition $(X_i,Y_i)$. We can label the vertices in $X_i$ by distinct AP-sets having $m_i$ elements and the common difference $d>1$ and label the vertices in $Y_i$ by distinct AP-sets having $n_i$ elements and the same common difference $d$,  where $m_i, n_i\ge 3$ are the positive integers such that $m_i+n_i-1=r$. Then, the corresponding IASI, say $f_i$, is an $r$-uniform IASI of $G_i$. Label all the vertices of every component of $G$ by distinct AP-sets having the same common difference $d$, as explained above, according to whether it is bipartite or not. Then, the function $f:V(G)\to \mathcal{P}(\mathbb{N}_0)$ defined by $f(v)=f_i(v)$, if $v\in V(G_i)$ is an isoarithmetic IASI of $G$, which is a uniform IASI of $G$. 
\end{proof}

Now that we have discussed the characteristics of arithmetic IASIs of certain graphs, all of whose elements have the same deterministic indices, we now proceed to consider the graphs whose different vertices have different deterministic indices.

\section{Biarithmetic IASI graphs}

By Theorem \ref{T-AIASI-g}, a graph admits an arithmetic IASI if and only if the deterministic ratios of all its edges are positive integers greater than or equal to $1$. We have considered the case when the deterministic ratio of all edges of $G$ is $1$. For studying the remaining cases, we introduce the following notion.

\begin{definition}{\rm
An arithmetic IASI $f$ of a graph $G$, under which the deterministic ratio of each edge of $G$ is a positive integer greater than $1$ and less than or equal to the set-indexing number of the end vertex of $e$ having smaller deterministic index.}
\end{definition}

In other words, a biarithmetic IASI of a graph $G$ is an arithmetic IASI $f$ of $G$, for which the deterministic indices of any two adjacent vertices $v_i$ and $v_j$ in $G$, denoted by $d_i$ and $d_j$ respectively such that $d_i < d_j$, holds the condition $d_j=kd_i$ where $k$ is a positive integer such that $1< k \le |f(v_i)|$.

In general, all edges of $G$ may not have the same deterministic ratio. Hence, we introduce the following notion.

\begin{definition}{\rm
Let $f$ be a biarithmetic IASI defined on a graph $G$. If the deterministic ratio of every edge of $G$ is the same, say $k$, then $f$ is called an {\em identical biarithmetic IASI} of $G$ and $G$ is called an {\em identical biarithmetic IASI graph}. }
\end{definition}

The existence of a biarithmetic IASI for a given graph $G$ depends upon the cardinality of set-labels of vertices of $G$ and the adjacency between the vertices. The following result establishes the admissibility of biarithmetic IASI by a given graph.

\begin{proposition}
Every graph $G$ admits a biarithmetic IASI.
\end{proposition}

The above result can be verified by taking the elements and cardinalities of the set-labels properly so that the conditions on the deterministic indices of the elements of $G$, as mentioned in \ref{T-AIASI-g}, are fulfilled.

An identical biarithmetic IASI may not exist for every graph $G$. The following theorem discusses the conditions required for a graph $G$ to admit an identical biarithmetic IASI.

\begin{theorem}\label{T-AIIBA1}
A graph $G$ admits an identical biarithmetic IASI if and only if it is bipartite.
\end{theorem}
\begin{proof}
Let $G$ be a bipartite graph having a bipartition $(X,Y)$ of $V(G)$. Now, define a function $f:V(G)\to \mathcal{P}(\mathbb{N}_0)$ in such a way that $f$ assigns distinct AP-sets having the same common difference, say $d>1$, to distinct vertices in $X$ and distinct AP-sets having the same common difference, say $k\,d$, to distinct vertices of $Y$, where$k=\min\{|f(u_i)|,u_i\in X\}$. Then, $f$ is an identical biarithmetic IASI of $G$.

Conversely, let $G$ admits an identical biarithmetic IASI. If possible, assume that $G$ is not a bipartite graph. Then, $G$ contains at least one odd cycle. For a positive integer $n=2i+1; i\ge 1$, let $C_{n}=v_1v_2v_3\ldots v_{2i+1}v_1$ be an odd cycle in $G$. Let $k$ be a positive integer such that $k\le |f(v_j)|$, where $1\le j \le n$ and $d$ be a positive integer, greater than $1$. Label the first vertex $v_1$ by an AP-set with common difference $d$. Now, label the vertices $v_2, v_3,\ldots,v_{2n}$ of $C_n$ by distinct AP-sets of non-negative integers in such a way that the edges connecting these vertices in $C_n$ have the deterministic ratio $k$. Then, the vertices of $C_n$ at the odd positions have the deterministic index $k^l\,d$, where $l$ is an even integer, positive or negative, and  the vertices of $C_n$ at the even positions have the deterministic index $k^s\,d$, where $s$ is an odd integer, positive or negative. 

Now, it remains to find a set-label for the vertex $v_{2i+1}$. If we choose an AP-set, which is not used for labeling the previous vertices, to label the vertex $v_{2i+1}$ in such a way that the edge $v_{2i}v_{2i+1}$ has the deterministic ratio $k$, then the deterministic index of the vertex $v_{2i+1}$ is $k^r\,d$, where $r$ is an even integer. Therefore, the deterministic ratio of the edge $v_{2i+1}v_1$ is greater than $k$.  If we choose a set-label for $v_{2i+1}$ in such a way that the edge $v_{2i+1}v_1$ has the deterministic ratio $k$, then the deterministic index of the vertex $v_{2i+1}$ is $k^r\,d$, where $r$ is an odd integer. But, we know that $v_{2n}$ is $k^{r_1}\,d$, where $r_1$ is also an odd integer. Therefore, the deterministic ratio of the edge $v_{2i}v_{2i+1}$ can not be $k$. In both cases, $C_n$ do not admit an identical biarithmetic IASI. Then, by Remark \ref{R-IBIASISG}, $G$ can not be an identical biarithmetic IASI graph, which is a contradiction to the hypothesis. Hence, $G$ must be bipartite. This completes the proof.
\end{proof}

In the following discussion, we study certain characteristics of identical and non-identical biarithmetic IASI graphs.

Analogous to Proposition \ref{P-APSL0}, we propose following result on biarithmetic IASI graphs. 

\begin{proposition}
Any subgraph of a biarithmetic IASI graph $G$ also admits a (induced) biarithmetic IASI. That is, existence of biarithmetic IASI is a hereditary property.
\end{proposition}

This proposition can be verified from the fact that an IASI of a graph $G$ induces an IASI to all its subgraphs. By the above proposition, it can be noted that an identical biarithmetic IASI of a graph $G$ also induces an identical biarithmetic IASI to any subgraph of $G$. This statement can also be re-stated as follows.

\begin{remark}\label{R-IBIASISG}{\rm
If a graph $G$ does not admit an identical biarithmetic IASI, then no supergraph of $G$ can be an identical biarithmetic IASI graph.}
\end{remark}

To learn about the set-indexing number of edges of a biarithmetic graph, we need the following theorem which estimates the set-indexing number of edges of a arithmetic IASI graph.

\begin{theorem}\label{T-AIASI1}
\cite{GS7} Let $G$ be a graph which admits an arithmetic IASI, say $f$ and let $v_i$ and $v_j$ be two adjacent vertices in $G$ with the deterministic indices $d_i$ and $d_j$, such that $d_i\le d_j$. Then, the set-indexing number of the edge $v_iv_j$ is $|f(v_i)|+k(|f(v_j)|-1)$, where $k\le |f(v_i)|$ is the deterministic ratio of the edge $v_iv_j$. 
\end{theorem}

The set-indexing number of the edges a biarithmetic IASI graph can be written as a special case of Theorem \ref{T-AIASI1} as follows.

\begin{theorem}\label{T-AIASI1a}
Let $G$ be a graph which admits an arithmetic IASI, say $f$ and let $v_i$ and $v_j$ be two adjacent vertices in $G$ with the deterministic indices $d_i$ and $d_j$, such that $d_j=k\,d_i$, where $k$ is a positive integer such that $1<k\le |f(v_i)|$. Then, the set-indexing number of the edge $v_iv_j$ is $|f(v_i)|+k(|f(v_j)|-1)$. 
\end{theorem}

Our next aim is to verify whether a biarithmetic IASI of a given graph can be a strong IASI of $G$. The following theorem explains a necessary and sufficient condition for a biarithmetic IASI of $G$ to be a strong IASI. 

\begin{theorem}\label{T-AIASI2}
Let $G$ be a graph which admits a biarithmetic IASI, say $f$. Then, $f$ is a strong IASI of $G$ if and only if the deterministic ratio of every edge of $G$ is equal to the set-indexing number of its end vertex having smaller deterministic index. 
\end{theorem}
\begin{proof}
Let $f$ be an arithmetic IASI of $G$. Let $v_i$ and $v_j$ are two adjacent vertices in $G$ and $d_i$ and $d_j$ be their deterministic indices under $f$. Without loss of generality, let $d_i<d_j$. Then, by Theorem \ref{T-AIASI1a}, the set-indexing number of the edge $v_iv_j$ is $|f(v_i)|+k(|f(v_j)|-1)$. 

Assume that $f$ is a strong IASI. Therefore, $f^{+}(v_iv_j)=mn$. Then,
\begin{eqnarray*}
|f(v_i)|+k(|f(v_j)|-1) & = & |f(v_i)|\,|f(v_j)|\\
\implies k(|f(v_j)|-1) & = & |f(v_i)|\,(|f(v_j)|-1)\\
\implies k & = & |f(v_i)|.
\end{eqnarray*}
Conversely, assume that the deterministic indices $d_i$ and $d_j$ of two adjacent vertices $v_i$ and $v_j$ respectively in $G$, where $d_i<d_j$ such that $d_j=|f(v_i)|.d_i$. Assume that $f(v_i)=\{a_r = a+rd_i:0 \le r < |f(v_i)|\}$ and $f(v_j)=\{b_s=b+s\,k\,d_i:0\le s < |f(v_j)|\}$, where $k\le |f(v_i)|$.  
Now, arrange the terms of $f^+(v_iv_j)=f(v_i)+f(v_j)$ in rows and columns as follows. For $b_s\in f(v_j), 0\le s < |f(v_j)|$, arrange the terms of $f(v_i)+b_s$ in $(s+1)$-th row in such a way that equal terms of different rows come in the same column of this arrangement. Then the common difference between consecutive elements in each row is $d_i$. Since $k=|f(v_i)|$, the difference between the final element of any row (other than the last row) and first element of its succeeding row is also $d_i$. That is, no column in this arrangement contains more than one element. Hence, all elements in this arrangement are distinct. Therefore, total number of elements in $f(v_i)+f(v_j)$ is $|f(v_i)|\,|f(v_j)|$. Hence, $f$ is a strong IASI. 
\end{proof}

Invoking Theorem \ref{T-AIASI2}, the condition for an identical biarithmetic IASI to be a strong IASI is established in the following theorem. 

\begin{theorem}\label{T-AIASI3}
An identical biarithmetic IASI of a graph $G$ is a strong IASI of $G$ if and only if one partition of $V(G)$ is $k$-uniformly set-indexed, where $k$ is the deterministic ratio of the edges of $G$.
\end{theorem}
\begin{proof} 
Let $f$ be an identical arithmetic IASI of $G$. Then, by Theorem \ref{T-AIIBA1}, $G$ is bipartite. Let $(X,Y)$ be the bipartition of $G$, where $X=\{u_i, 1\le i \le r\}$ and $Y=\{v_j, 1\le j \le s\}$, $r+s=|V(G)|$. Also let $d_i$ be the deterministic index of the vertex $u_i \in X$ and $d'_j$ be the deterministic index of $v_j \in Y$. Without loss of generality, let $X$ is $k$-uniformly set-indexed, where $k$ is the deterministic ratio of the edges of $G$. Therefore, $|f(u_i)|=k~~ \forall u_i \in X$. Then, since $f$ is an identical biarithmetic IASI, we have $d'_j=|f(u_i)|\,d_i$ for every edge $u_iv_j \in V(G)$. Hence, by Theorem \ref{T-AIASI2}, $f$ is a strong IASI of $G$.

Conversely, assume that the identical biarithmetic IASI $f$ of $G$ is a strong IASI. Then, for every edge $e$ of $G$, the set-label of the end vertex of $e$ having smaller deterministic index must have exactly $k$ elements, where $k$ is the deterministic ratio of the edges of $G$. Let $X$ be the set of all these vertices having set-indexing number $k$. Since $f$ is an identical biarithmetic IASI, no vertices in $X$ can be adjacent to each other. Therefore, $(X, V-X)$ is a bipartition of $V(G)$, where $X$ is $k$-uniformly set-indexed. This completes the proof.  
\end{proof}

In this context, it is interesting to check the existence of saturated classes or maximal compatibility classes and their cardinalities. The following theorem provides the necessary and sufficient condition  for the existence of saturated classes and the number of saturated classes in the Cartesian product of the set-labels of two adjacent vertices.

\begin{theorem}\label{T-NSC-II}
Let $G$ be a graph that admits a biarithmetic IASI, say $f$. Let $v_i$ and $v_j$ be two adjacent vertices in $G$, where $v_i$ has the smaller deterministic index. Let $k$ be the deterministic ratio of the edge $v_iv_j$. Then, a compatible class in $f(v_i)\times f(v_j)$ is a saturated class if and only if $|f(v_i)|=(|f(v_j)|-1)\,k+r, ~r>0$. Also,  number of saturated classes in $f(v_i)\times f(v_j)$ is $|f(v_i)|-(|f(v_j)|-1)\,k$. Moreover, for $1\le p \le n-1$, there are exactly $2k$ compatibility classes contain $p$ elements.
\end{theorem}
\begin{proof}
Let $f(v_i)=\{a, a+d_i,a+2d_i,\ldots,a+(m-1)d_i\}$ and $f(v_j)=\{b, b+kd_i,b+2kd_i,\ldots,b+(n-1)kd_i\}$, where $a$ and $b$ are positive integers. Consider the set-label of the edge $v_iv_j$ defined by $f^{+}(v_iv_j)=f(v_i)+f(v_j)$. Let $a+b=q$. Then, $f^{+}(v_iv_j)=\{q,q+d,q+2d,\ldots,q+[(m-1)+k(n-1)]d\}$. 

Arrange the elements of $f(v_i)\times f(v_j)$ in rows and columns as follows. Write the elements of $f(v_i)+\{b_s\},~b_s\in f(v_j),~0\le s\le (n-1)$ in $(s+1)$-th row in such a way that equal terms in these rows come in the same column. Hence, we have $n$ rows containing $m$ elements in each row. Then, each column of this arrangement corresponds to a compatibility class and the number of elements in a column is the cardinality of the corresponding compatibility class. It is to be noted that the last $(m-k)$ elements of each row, except the last row, will be the first $m-k$ elements of the succeeding row. Hence, for $j\le n$, if $m>jk$, then last $(m-(j-1)k)$ elements of the first row will be the first $(m-(j-1)k)$ elements of the $j$-th row.

Assume that there are $r>0$ saturated classes in $f(v_i)\times f(v_j)$. Then, clearly $m>n$ and hence by Lemma \ref{L-CardCC}, a saturated class in $f(v_i)\times f(v_j)$ can have a maximum of $n$ elements. Then, $r$ columns of the arrangement contains $n$ elements. Therefore, the last $m+(n-1)k$ elements of the first row are the first $m+(n-1)k$ elements in the $n$-th (the last) row. That is, $m-(n-1)k=r$ or  $m=(n-1)k+r$ where $r$ is a positive integer.

Conversely, assume that $m=(n-1)k+r, r>0$.  From the above step, we note that $m-(n-1)k=r$ elements of the first row are common to all $n$ rows in the above row and column arrangement. That is, $r$ elements are common to all the $n$ rows of this arrangement. Hence, there are $r=m-(n-1)k$ saturated classes in $f(v_i)\times f(v_j)$. This completes the proof.

Since each column in the row and column arrangement, we mentioned above, corresponds to a compatibility class and the set-indexing number of an edge $v_iv_j$ is equal to the number of distinct compatibility classes in $f(v_i)\times f(v_j)$, by Theorem \ref{T-AIASI1a}, we have $|f(v_i)|+k(|f(v_j)|-1)$ columns in the arrangement. Note that the first $k$ elements of the first row and the last $k$ elements of the last row do not appear in any other rows. Therefore, the number columns having exactly one element is $2k$. That is, the number of compatibility classes with one element is $2k$.

Now remove these $2k$ columns from the arrangement. Then, in the revised arrangement, the first $k$ elements of the first two rows are the same and the last $k$ elements of the last two rows are the same and these element do not appear in any other rows. Therefore, the number of compatibility classes with $2$ elements is $2k$. 

Proceeding like this, we have the number of compatibility classes having $p$ elements is $2k$, where $1\le p \le (n-1)$.
\end{proof}

It is clear that if $f(v_i)<f(v_j)$ in Theorem \ref{T-NSC-II}, then there is no saturated class in $f(v_i)\times f(v_j)$. Then, our next intention is to study the case when $|f(v_i)|=pk+q$, where $p$ and $q$ are non-negative integers such that $p<(|f(v_j)|-1)$ and $q<k$. Hence, we need to study further to determine the number of maximal compatibility classes and their cardinality. The following theorem provides the number of maximal compatibility classes in $f(v_i)\times f(v_j)$.

\begin{theorem}\label{T-NMCC-II}
Let $G$ be a graph that admits a biarithmetic IASI, say $f$. Let $v_i$ and $v_2$ be two adjacent vertices of $G$, where $v_i$ has the smaller deterministic index and $k\le|f(v_i)|$, be the deterministic ratio of the edge $v_iv_j$. If $|f(v_i)|=pk+q$, where $p,q$ are non-negative integers such that $p\le (|f(v_j)|-1)$ and $q<k$, then
\begin{enumerate}
\item[(i)] if $q=0$, then $(|f(v_j)|-p+1)k$ compatibility classes are maximal compatibility classes and contain $p$ elements.
\item[(ii)] if $q>0$, then $(|f(v_j)|-p-1)k+q$ compatibility classes are compatibility classes and contain $(p+1)$ elements.
\end{enumerate}
\end{theorem}
\begin{proof}
Let $v_i$ and $v_j$ be two adjacent vertices of $G$ with deterministic indices $d_i$ and $d_j$, such that $d_j=k.d_i$,  where $k\le |f(v_i)|$ is the deterministic ratio of the edge $v_iv_j$. Also, let $|f(v_i)|=m$ and $|f(v_j)|=n$.  Since $f$ is a biarithmetic IASI, we have $k\le m$. By Theorem \ref{T-APSL3}, the set-indexing number of the edge $v_iv_j$ is $m+k(n-1)$.  Now, assume that $m=pk+q$, where $p,q$ are non-negative integers such that $1\le p\le (n-1)$ and $0\le q<k$.  Arrange the elements of $f(v_i)\times f(v_j)$ in such a way that $f(v_i)+\{b_s\}$, where $b_s=b+sd\in A_j, 0\le s\le k(n-1)\}$, in $(s+1)$-th row and equal terms in these rows come in the same column.

\noindent {\em Case-1:} Let $q=0$. That is, $m= pk$. From the above arrangement, we observe that the last $k$ elements of the each row in the first half of the arrangement and the first $k$ elements of each row in the second half of this arrangement are common to exactly $p$ rows. Therefore, the cardinality of a maximal class in this case is $p$.

Moreover, as explained in Theorem \ref{T-NSC-II}, for $1\le j \le p-1$ there exist exactly $2k$ classes containing $j$ elements. Therefore, the total number of non-maximal compatibility classes is $2k(p-1)$. Therefore, the number of maximal compatibility classes is $m+k(n-1)-2k(p-1)=pk+(n-1)k-2pk=(n-p+1)k$.

\noindent {\em Case-2:} Let $q\ge 0$. That is, $m=pk+q$. From the above arrangement, we observe that the last $q$ elements of the each row in the first half of the arrangement and the first $k$ elements of each row in the second half of this arrangement are common to exactly $p+1$ rows. Therefore, the cardinality of a maximal class in this case is $p+1$.
Now, for $1\le j \le p$ there exist exactly $2k$ classes containing $j$ elements. Therefore, the total number of non-maximal compatibility classes is $2kp$. Therefore, the number of maximal compatibility classes is $m+k(n-1)-2kp = pk+q+(n-1)k-2pk=(n-p-1)k+q$.
That is, the number of maximal classes here is $(n-p-1)k+q$ and the number of elements in each of these maximal classes is $p+1$. This completes the proof.
\end{proof}

\section{Conclusion}

In this paper, we have discussed some characteristics of graphs which admit certain types of IASIs called isoarithmetic and biarithmetic IASIs. We have formulated some conditions for some graph classes to admit these types of arithmetic IASIs and discussed about certain properties characteristics of isoarithmetic and biarithmetic IASI graphs. Problems related to the characterisation of different biarithmetic IASI graphs are still open. The problems regarding the admissibility of certain graph operations and products which admit isoarithmetic and biarithmetic IASIs, characterisation of given graphs which admit biarithmetic IASIs, uniform and non-uniform, etc. are promising and worth studying. The IASIs which are vertex arithmetic, but not edge arithmetic can also be studied in detail.

The IASIs under which the vertices of a given graph are labeled by different standard sequences of non negative integers, are also note worthy.   The problems of establishing the necessary and sufficient conditions for various graphs and graph classes to have certain IASIs still remain unsettled. All these facts highlight a wide scope for further studies in this area.

\end{document}